\documentclass{article}

\usepackage[latin1]{inputenc}
\usepackage[english]{babel}
\usepackage{times}
\usepackage{epsfig}
\usepackage{amsfonts}
\usepackage{amsthm}
\usepackage{float}
\usepackage{amsmath}
\usepackage{amssymb}
\usepackage{latexsym}
\usepackage{graphicx, color}
\usepackage{caption}
\usepackage{subcaption}
\usepackage{prodint}
\usepackage[round]{natbib}
\usepackage{bm}
\usepackage{textcomp}
 \usepackage{paralist}
\usepackage{tikz}
\usepackage{mathabx}
\usetikzlibrary{arrows,automata,positioning}

\newcommand{\rnc}{\renewcommand}
\newcommand{\nc}{\newcommand}
\newcommand{\mrm}{\mathrm}
\renewcommand{\b}{\textbf}
\newcommand{\bs}{\boldsymbol}
\nc{\mb}{\mathbb}
\nc{\mac}{\mathcal}
\nc{\E}{\mb{E}}
\nc{\N}{\mb{N}}
\nc{\R}{\mb{R}}
\nc{\Q}{\mb{Q}}
\rnc{\P}{\mrm P}
\rnc{\d}{\mrm d}
\nc{\C}{\mac{C}}
\nc{\D}{\mac{D}}
\nc{\B}{\mac{B}}
\nc{\wh}{\widehat}
\nc{\oPo}{\stackrel{p}{\longrightarrow}}
\nc{\oWo}{\stackrel{w}{\longrightarrow}}
\nc{\oDo}{\stackrel{d}{\longrightarrow}}
\nc{\tn}{\textnormal}

\relax 
\allowdisplaybreaks[3]

\input tex.mac
\input AMENGL.MAC
\rnc{\mc}{\mathcal}
\numberwithin{equation}{section}

\newtheorem{thm}{Theorem}
\newtheorem{rem}{Remark}

\begin{document}

    
\title{\Large Bootstrapping the Kaplan-Meier Estimator on the Whole Line
}
    
\author{Dennis Dobler$^{*}$ 
}
\maketitle

\begin{abstract}

This article is concerned with proving the consistency of Efron's (1981) bootstrap 
for the Kaplan-Meier estimator on the whole support of a survival function.
While other works address the asymptotic Gaussianity of the estimator itself without restricting time
(e.g. \citealp{gill83}, and \citealp{ying89}),
we enable the construction of bootstrap-based time-simultaneous confidence bands for the whole survival function.
Other practical applications include bootstrap-based confidence bands for the mean residual life-time function or the Lorenz curve as well as confidence intervals for the Gini index.

\end{abstract}


\vfill
\vfill

\noindent${}^{*}$ {University of Ulm, Institute of Statistics, Germany}\\

\newpage

\section{Introduction}

This article reconsiders Efron's~(1981) \nocite{efron81} bootstrap of Kaplan-Meier estimators.
It is well-known that drawing with replacement directly from the original observations consisting of (event time, censoring indicator) reproduces the correct covariance structure;
see e.g. \cite{akritas86}, \cite{lo86}, \cite{horvath87} or \cite{vaart96} for an application in empirical processes.
Let $T: \Omega \rightarrow (0,\tau)$ be a continuously distributed random survival time with survival function given by $S(t)  = 1 - F(t) = P(T > t)$.
For conceptual convenience we mainly refer to $T$ as a random \emph{survival time},
although other interpretations are also reasonable; see the examples below.
In the previously mentioned articles the typical assumption $S(\tau) > 0$ is met for mathematical convenience in proving weak convergence of estimators for $S$ on the Skorohod space $D[0,\tau]$
and because most studies involve a rather strict censoring mechanism:
after a pre-specified end of study time each individual without an observed event is considered as right-censored.
Thus, it is often not possible to draw inference on functionals of the whole survival function.

Some functionals, however, indeed require the possibility to observe arbitrarily large survival times.
For instance, consider the \emph{mean residual life-time} function 
\begin{align}
 \label{eq:MRLT}
 t \longmapsto g(t) = \E[T - t  \ | \ T > t] = \frac{1}{S(t)} \int_t^\tau S(u) \d u;
\end{align}
see e.g. \cite{meilijson72}, \cite{gill83}, Remark~3.3, and \cite{stute93}.
This function describes the expected remaining life-time given the survival until a point of time $t > 0$.
Another, econometric example of a functional of whole survival curves is the \emph{Lorenz curve} 
\begin{align}
\label{eq:Lorenz}
 p \longmapsto L(p) = \mu^{-1} \int_0^p F^{-1}(t) \d t = \mu^{-1} \int_0^{F^{-1}(p)} s \d F(s),
\end{align}
where $F^{-1}(t) = \inf \{ u \geq 0: F(u) \geq t \}$ is the left-continuous generalized inverse of $F$ and $\mu = \int_0^\infty t \d F(t)$ is its mean.
With the interpretation of $T$ being the income of a random individual in a population,
this function $L$ obviously represents the total income of the lowest $p$th fraction of all incomes.
A closely related quantity is the \emph{Gini index} 
\begin{align}
\label{eq:Gini}
 G = \frac{\int_0^1 (u - L(u)) \d u}{\int_0^1 u \d u} \in [0,1]
\end{align}
as a measure of uniformity of all incomes within a population; see e.g. \cite{tse06}.
The value $G = 0$ represents perfect equality of all incomes,
whereas $G=1$ describes the other extreme: 
only one persons gains everything and the rest nothing.

All quantities \eqref{eq:MRLT}, \eqref{eq:Lorenz} and \eqref{eq:Gini} 
are statistical functionals of the \emph{whole} survival function $S$.
First analyzing $S$ only on a subset of its support results inevitably in an alternation of the above functionals in a second step.
And this affects the interpretation of such quantities.
In order to circumvent such problems,
estimating the whole survival function is the obvious solution:
Henceforth, denote by $\tau = \inf\{t \geq 0 : S(t) = 0 \} \in (0,\infty]$ the support's right end point.
\cite{wang87} and \cite{stute93} showed the uniform consistency of the \emph{Kaplan-Meier} or \emph{product-limit estimator} 
$(\wh S(t))_{t \in [0,\tau]}$ for $(S(t))_{t \in [0,\tau]}$ 
and \cite{gill83} and \cite{ying89} proved its weak convergence on the Skorohod space $D[0,\tau]$.
For robust statistical inference procedures concerning the above functionals of $S$
it is thus necessary to extend well-known bootstrap results for the Kaplan-Meier estimator to the whole Skorohod space $D[0,\tau]$.
After presenting this primary result
we deduce inference procedures for the quantities \eqref{eq:MRLT} to \eqref{eq:Gini}.

This article is organized as follows.
Section~\ref{sec:EBS_KME_PRE} introduces all required estimators, recapitulates previous weak convergence results on $D[0,t]$, $t < \tau$, and provides handy results for checking all main assumptions.
The main theorems on weak convergence of the bootstrap Kaplan-Meier estimator are presented in Section~\ref{sec:EBS_KME_WL_MAIN}, 
including a consistency theorem for a bootstrap variance function estimator.
Section~\ref{sec:EBS_KME_APP} deduces inference procedures for \eqref{eq:MRLT} to \eqref{eq:Gini}
and the final Section~\ref{sec:EBS_KME_DISC} gives a discussion on future research possibilities.
All proofs are given in the Appendix.
Most of this article's results originate from the University of Ulm PhD thesis of Dennis Dobler; 
cf. Chapter~6 and~7 of \cite{dobler16}.

\section{Preliminary Results}
\label{sec:EBS_KME_PRE}

Let $T_1, \dots, T_n: \Omega \rightarrow (0,\infty), n \in \N,$ be independent survival times 
with continuous survival functions
$S(t) = 1 - F(t) = P(T_1 > t)$ and cumulative hazard function $A(t) = \int_0^t \alpha(u) \d u =  -\int_0^t (\d S) / S_- =  -\log S(t)$.
Independent thereof, let $C_1, \dots, C_n: \Omega \rightarrow (0,\infty)$ be i.i.d. (censoring) random variables with (possibly discontinuous) survival function
$G(t) = P(C_1 > t)$ such that the observable data consist of all $1 \leq i \leq n$ pairs $(X_i, \delta_i) := (T_i \wedge C_i, \b 1\{X_i = T_i\})$.
Here, $\b 1 \{ \cdot \}$ is the indicator function.
Thus, the survival function of $X_1$ is $ H = S \cdot G$.
The Kaplan-Meier estimator is defined by $\wh S_n(t) = \prod_{i: X_{i:n} \leq t} (1 - \frac{\delta_{[i:n]}}{n-i+1})$,
where $(X_{1:n},\dots, X_{n:n})$ is the order statistic of $(X_1, \dots, X_n)$
and $(\delta_{[1:n]}, \dots, \delta_{[n:n]})$ are their concomitant censoring indicators.
Throughout, we assume that
\begin{align}
\label{cond:weak_cens}
 - \int_0^\tau \frac{\d S}{G_-} < \infty
\end{align}
which restricts the magnitude of censoring to a reasonable level.
For instance \cite{gill83}, \cite{ying89} and \cite{akritas97} require this condition for an analysis of the large sample properties of Kaplan-Meier estimators on the whole support $[0,\tau]$.
Thereof, \cite{gill83} requires Condition~\eqref{cond:weak_cens} for a vanishing upper bound in Lenglart's inequality.
Obviously, the above condition implies that $[0,\tau]$ is contained in the support of $G$;
see also \cite{allignol14} for a similar condition in a non-Markov illness-death model, reduced to a competing risks problem.

Denote by $\widehat T_n := \max_{i \leq n}  X_i$ the largest observed event or censoring time
and let, for a function $t \mapsto f(t)$, the notation ${f}^{\widehat T_n}$ be its stopped version,
i.e., $f^{\widehat T_n}(t) = f(t \wedge \widehat T_n)$.
The monotone function
$
 t \mapsto \sigma^2(t) =  \int_0^t (\d A) / H_-
$
is the asymptotic variance function of the related Nelson-Aalen estimator for $A$
and reappears in the asymptotic covariance function of $\wh S_n$.
Throughout, all convergences (in distribution, probability, or almost surely) are understood to hold as $n \rightarrow \infty$ 
and convergence in distribution and in probability 
are denoted by $\stackrel{d}{\rightarrow}$ and $\stackrel{p}{\rightarrow}$, respectively.
The present theory relies on the following weak convergence results for the Kaplan-Meier process $\widehat S_n$ of $S$.
\begin{lemma}
\label{lem:KME_WL}
Let $B$ denote a Brownian motion on $[0,\tau]$ and suppose \eqref{cond:weak_cens} holds. \\
 \tn{(a)} \tn{Theorem 1.2(i) of \citep{gill83}:}
 On $D[0,\tau]$ we have
  $\sqrt{n} (\widehat S_n - S)^{\widehat T_n} \oDo W:= S \cdot  (B \circ \sigma^2),$ \\
 \tn{(b)} \tn{Part of Theorem 2 in \citep{ying89}:}
 On $D[0,\tau]$ we have
  $\sqrt{n} (\widehat S_n - S) \oDo W = S \cdot (B \circ \sigma^2).$
\end{lemma}
Denote by $\widehat A_n(t) = \sum_{i: X_{i:n} \leq t} \frac{\delta_{[i:n]}}{n-i+1}$ the Nelson-Aalen estimator for the cumulative hazard function $A(t)$ 
and by $\widehat G_n$ the Kaplan-Meier estimator for the censoring survival function $G$.
Note that $\widehat H_n = \widehat G_n \widehat S_n$ holds for the empirical survival function of $H$
since, almost surely (a.s.), no survival time equals a censoring time: $T_i \neq C_j$ a.s. for all $i,j$.
The asymptotic covariance function $\Gamma$ of $W$ in Lemma~\ref{lem:KME_WL} 
and a natural estimator $\widehat \Gamma_n$ are given by
\begin{align*}
 \Gamma(u,v) = S(u) \Big( \int_0^{u \wedge v} \frac{\d A}{H_-} \Big) S(v) \quad \text{and} \quad
 \widehat \Gamma_n(u,v)  = \widehat S_n(u) \Big( \int_0^{u \wedge v} \frac{\d \widehat A_n}{\widehat H_{n-}} \Big) \widehat S_n(v).
\end{align*}
The following lemma is helpful for an assessment of Condition~\eqref{cond:weak_cens} and for studentizations.
\begin{lemma}
\label{lem:VAR_KME_WL}
 \tn{(a)} For all $t \in [0,\tau]$ it holds that
  $$ - \int_t^\tau \frac{\d \widehat S_n}{\widehat G_{n-}} \oPo - \int_t^\tau \frac{\d S}{G_-} \leq \infty. $$ \\
 \tn{(b)} In case of \eqref{cond:weak_cens} we have
   \begin{align*}
      \sup_{(u,v) \in [0, \tau]^2} | \widehat \Gamma_n (u,v) - \Gamma(u,v) | \oPo 0.
   \end{align*}
\end{lemma}
%
%

\section{Main Results}
\label{sec:EBS_KME_WL_MAIN}

The limit distribution of the Kaplan-Meier process in Lemma~\ref{lem:KME_WL} shall be assessed via bootstrapping.
To this end, we independently draw $n$ times with replacement from $(X_1,\delta_1), \dots, (X_n,\delta_n)$
and denote the thus obtained bootstrap sample by $(X_1^*, \delta_1^*), \dots, (X_n^*,\delta_n^*)$.
Throughout, denote by $\Gamma^*_n$, $S^*_n$ etc. the obvious estimators but based on the bootstrap sample.
Note that this requires a discontinuous extension of the above quantities.
The following theorem is the basis of all later inference methods.
\begin{thm}
\label{thm:ebs_kme_wl}
Let $B$ denote a Brownian motion on $[0,\tau]$ and suppose that \eqref{cond:weak_cens} holds.
 Then we have, conditionally on $X_1, X_2, \dots $,
 $$ \sqrt{n} (S_n^* - \widehat S_n) \oDo W = S \cdot (B \circ \sigma^2) $$
 on $D[0,\tau]$ in probability.
\end{thm}
Many statistical applications involve a consistent variance estimator,
e.g. Hall-Wellner or equal precision confidence bands for $S$; cf. \cite{abgk93}, p. 266.
In order to asymptotically reproduce the same limit on the bootstrap side,
the uniform consistency of a bootstrapped variance estimator (defined on the whole support $[0,\tau]^2$ of the covariance function) needs to be verified.
To this end, introduce the bootstrap version of $\wh \Gamma_n$, that is,
$$ \Gamma_n^*(u,v) = S_n^*(u) \Big( \int_0^{u \wedge v} \frac{\d A_n^*}{H_{n-}^*} \Big) S_n^*(v). $$
For all $\varepsilon > 0$, its uniform consistency (here and below always meaning 
conditional convergence in probability given $X_1, X_2, \dots$ in probability)
over all points $ (u,v) \in [0,\tau]^2 \setminus [\tau-\varepsilon, \tau]^2 $
is an immediate consequence of Theorem~\ref{thm:ebs_kme_wl} in combination with the continuous mapping theorem:
Write the absolute value of the integral part minus its estimated counterpart as
\begin{align*}
  & \Big| \int_0^{u \wedge v} \frac{(\widehat H_{n-} - H^*_{n-})\d A_n^* - H^*_{n-}\d (\widehat A_n - A_n^*)}{H_{n-}^* \widehat H_{n-}} \Big| \\
  & \leq \frac{\sup_{(0,u\wedge v)}|\widehat H_{n} - H^*_{n}|}{H^*_{n}((u \wedge v)-) \widehat H_{n}((u \wedge v)-)}  A_n^*(u \wedge v)
  + \Big| \int_0^{u \wedge v} \frac{\d (\widehat A_n - A_n^*)}{\widehat H_{n-}} \Big|.
\end{align*}
The first term is asymptotically negligible due to P\`olya's theorem
and the second term becomes small due to the continuous mapping theorem applied to the integral functional
and the logarithm functional.
Here the restriction to $[0,\tau]^2 \setminus [\tau-\varepsilon, \tau]^2 $ simplified the calculations 
since all denominators are asymptotically bounded away from zero.

For uniform consistency on the whole rectangle $[0,\tau]^2$, however, 
similar arguments as for the bootstrapped Kaplan-Meier process on $[0,\tau]$ are required.
Compared to~\eqref{cond:weak_cens}, we postulate a slightly more restrictive censoring condition.

\begin{lemma}
 \label{lem:ebs_kme_var_wl}
 Suppose that
 \begin{align}
 \label{cond:weaker_cens}
  - \int_0^\tau \frac{\d S}{G_-^3} < \infty.
 \end{align}
Then we have the following conditional uniform consistency given $X_1, X_2, \dots$ in probability:
   \begin{align}
   \label{eq:unif_conv_cov_est_WL}
      \sup_{(u,v) \in [0, \tau]^2} | \Gamma^*_n (u,v) - \widehat \Gamma_n(u,v) | \oPo 0 \quad \text{in probability as } n \rightarrow \infty.
   \end{align}
\end{lemma}

\begin{rem}
 The proof of Lemma~\ref{lem:ebs_kme_var_wl} shows that Condition~\eqref{cond:weaker_cens}
 can be diminished to
 \begin{align*}
  - \int_0^\tau \frac{\d S}{G_-} - \int_0^\tau  \frac{S^{1-\delta} \d S}{G_-^{2+ \delta}} < \infty
 \end{align*}
 for some $\delta \in (0,1)$.
 This is due to the inequality $(n \widehat H_{n-}^3)^{-1} \leq ( n^\delta \widehat H_{n-}^{2+\delta})^{-1} $.
\end{rem}

\section{Applications}
\label{sec:EBS_KME_APP}

Applications of Theorem~\ref{thm:ebs_kme_wl} concern confidence intervals for the mean residual life-time $g(t) = \E[T-t \ | \ T > t ]$ on compact sub-intervals $[t_1,t_2] \subset [0,\tau)$ in case of $\tau < \infty$ as well as confidence regions for the Lorenz curve $L$ and the Gini index $G$.
To this end, we apply the functional delta-method (e.g. \citealp{abgk93}, Theorem~II.8.1) which in turn requires the Hadamard-differentiability of all involved statistical functionals.

\subsection*{Confidence Bands for the Mean-Residual Lifetime Function}

Let $0 \leq t_1 \leq t_2$ and introduce the space $C[t_1,\tau]$ of continuous functions on $[t_1,\tau]$ equipped with the supremum norm
as well as the subset 
$$\widetilde C[t_1,t_2] = \{ f \in C[t_1,\tau] : \inf_{s \in [t_1,t_2]} |f(s)| > 0 \} \subset C[t_1,\tau]$$
containing all continuous functions having a positive distance to the constant zero function on the interval $[t_1,t_2]$.
Similarly, let 
$$\widetilde D[t_1,t_2] = \{ f \in D[t_1,\tau] : \inf_{s \in [t_1,t_2]} |f(s)| > 0, \sup_{s \in [t_1,\tau]} |f(s)| < \infty \} \subset D[t_1,\tau]$$
be the extension of $\widetilde C[t_1,t_2]$ to possibly discontinuous, bounded c\`adl\`ag functions.
For the notion of Hadamard-differentiability tangentially to subsets of $D[t_1,\tau]$,
see Definition~II.8.2, Theorem~II.8.2 and Lemma~II.8.3 in \cite{abgk93}, p. 111f.
The following lemma makes the functional delta-method available for applications to the mean residual life-time function.
\begin{lemma}
\label{lem:hadamard_ex}
    Let $\tau < \infty$ and $[t_1,t_2] \subset [0,\tau)$ be a compact interval. 
    Then $$\psi : \widetilde D [t_1,t_2] \rightarrow D[t_1,t_2], 
    \quad \theta(\cdot) \mapsto \frac{1}{\theta(\cdot)} \int_{\cdot}^\tau \theta(s) \d s$$
    is Hadamard-differentiable at each $\theta \in \widetilde C[t_1,t_2]$ tangentially to $C^2[t_1,\tau]$
    with continuous linear derivative $\d \psi(\theta) \cdot h \in D[t_1,t_2]$ given by
 \begin{align*}
  ( \d \psi(\theta) \cdot h )(s) := \frac{1}{\theta(s)} \int_s^\tau h(u) \d u - h(s) \int_s^\tau \frac{\theta(u)}{\theta^2(s)} \d u.
 \end{align*}
\end{lemma}
As pointed out in \cite{gill89} or \cite{abgk93}, p. 110,
the functional delta-method is established on the functional space $D[t_1,\tau]$ (or subsets thereof) equipped with the supremum norm.
However, in case of limiting processes with continuous sample paths,
``weak convergence in the sense of the [Skorohod] metric and in the sense of the supremum norm are exactly equivalent'' \citep{abgk93}. See also Problem 7 in~\cite{pollard84}, p. 137.
The convergence result of Theorem~\ref{thm:ebs_kme_wl} combined with the functional $\psi$ of Lemma~\ref{lem:hadamard_ex} constitutes the following weak convergence.
\begin{lemma}
\label{lem:EBS_KME_APP}
 Suppose that~\eqref{cond:weak_cens} holds.
 On the Skorohod space $D[t_1,t_2]$ we then have
 \begin{align*}
  \sqrt{n} \Big( \int_\cdot^\tau \frac{\widehat S_n(u)}{\widehat S_n(\cdot)} \d u 
    - \int_\cdot^\tau \frac{S(u)}{S(\cdot)} \d u \Big) \oDo U
 \end{align*}
 and, given $X_1, X_2,\dots$,
 \begin{align*}
  \sqrt{n} \Big( \int_\cdot^\tau \frac{S_n^*(u)}{S_n^*(\cdot)} \d u 
    - \int_\cdot^\tau \frac{\widehat S_n(u)}{\widehat S_n(\cdot)} \d u \Big) \oDo U
 \end{align*}
  in outer probability. The Gaussian process $U$ has a.s. continuous sample paths,
  mean zero and covariance function
  \begin{align*}
   (r,s) \mapsto \int_{r \vee s}^\tau \int_{r \vee s}^\tau \frac{\Gamma(u,v)}{S(r)S(s)} \d u \d v
   - \sigma^2(r \vee s) g(r) g(s),
  \end{align*}
  where $g(t) = \E[T_1 - t \ | \ T_1 > t] =  \int_t^\tau \frac{S(u)}{S(t)} \d u$ is again the mean residual life-time function. 
\end{lemma}
The previous lemma in combination with the continuous mapping theorem 
almost immediately gives rise to the construction of asymptotically valid confidence regions 
for the mean residual life-time function.
According to the functional delta-method we may first apply, 
e.g. an $\arcsin$- or $\log$-transformation to ensure that only positive values are included in the confidence regions;
cf. Section~IV.1.3 in \cite{abgk93}, p. 208ff.
For ease of presentation, only the linear regions are stated below.
\begin{thm}
\label{thm:EBS_KME_APP}
Let $0 \leq t_1 \leq t_2 < \tau$.
  Choose any $\alpha \in (0,1)$ and suppose that~\eqref{cond:weak_cens} holds.
  An asymptotic two-sided $(1-\alpha)$-confidence band for the mean residual life-time function $(\E[T_1 - t \ | \ T_1 > t])_{t \in [t_1,t_2]} $ is given by
  \begin{align*}
   \Big[ \int_t^\tau \frac{\widehat S_n(u)}{\widehat S_n(t)} \d u - \frac{q_{n_1,n_2}^{MRLT}}{\sqrt n}, 
    \int_t^\tau \frac{\widehat S_n(u)}{\widehat S_n(t)} \d u + \frac{q_{n_1,n_2}^{MRLT}}{\sqrt n} \Big]_{t \in [t_1,t_2]}
  \end{align*}
  where $q^{MRLT}_{n_1,n_2}$ is the $(1-\alpha)$-quantile of the conditional law given $(X_{1},\delta_{1}), \dots, (X_{n},\delta_{n})$ of 
  \begin{align*}
    \sqrt{n} \sup_{t \in [t_1,t_2]} \Big| \int_t^\tau \frac{S_n^*(u)}{S_n^*(t)} \d u 
    - \int_t^\tau \frac{\widehat S_n(u)}{\widehat S_n(t)} \d u  \ \Big|.
  \end{align*}
\end{thm}
\begin{rem}
\tn{(a)}
 Instead of using a transformation as indicated above Theorem~\ref{thm:EBS_KME_APP},
 one could also employ a studentization using $\widehat \Gamma_n$ and $\Gamma_n^*$.
 Plugging these and consistent estimators for the other unknown quantities into the asymptotic variance representation
 yields consistent variance estimators for the statistic of interest.
 This yields a Gaussian process with asymptotic variance 1 at all points of time for the mean residual life-time estimates. \\
\tn{(b)} In practice, the construction of confidence bands for the mean residual life-time function
 requires to choose $t_2$ depending on the data:
  else, too large choices of $t_2$ might result in $\widehat S_n(t_2) = 0$, in which case the above estimator would not be well-defined.
\end{rem}

\subsection*{Confidence Regions for the Lorenz Curve and the Gini Index}

Suppose $S$ has compact support $[0,\tau]$, i.e. let this again be the smallest interval satisfying $S(0)=1$ and $S(\tau)=0$. 
As estimators for the Lorenz curve and the Gini index we consider the plug-in estimates
$$\widehat L_n(p) = \frac1{\widehat\mu_n} \int_0^p (1- \widehat S_n(t))^{-1} \d t 
\quad \text{and} \quad \wh G_n = \frac{\int_0^1 (u - \widehat L_n(u)) \d u}{\int_0^1 u \d u},$$
where $\widehat \mu_n = \int_0^{\tau} s \d \widehat S_n(s)$.
The restricted and unscaled Lorenz curve estimator under independent right-censoring has been bootstrapped by \cite{horvath87}.
\cite{tse06} discussed the large sample properties of the above Lorenz curve estimator (even under left-truncation)
and also of the normalized estimated Gini index
\begin{align*}
 \sqrt{n} (\widehat G_n - G)  = \sqrt{n} \Big( \frac{\int_0^1 (u - \widehat L_n(u)) \d u}{\int_0^1 u \d u} - \frac{\int_0^1 (u - L(u)) \d u}{\int_0^1 u \d u} \Big) 
  = 2\sqrt{n} \int_0^1 (L(u) - \widehat L_n(u))\d u.
\end{align*}
Again equip all subsequent function spaces with the supremum norm.
Let $D_\uparrow[0,\tau] \subset D[0,\tau]$ be the set of all distribution functions on $[0,\tau]$ with no atom in $0$, and let $D_-[0,\tau]$ be the set of all c\`agl\`ad functions on $[0,\tau]$.
First, we consider the normalized estimated Lorenz curve, i.e.
the process $W_n: \Omega \rightarrow [0,1]$ given by
\begin{align*}
  W_n(p) = & \sqrt{n} \Big( \frac1{\widehat \mu_n} \int_0^p (1- \widehat S)^{-1}(s) \d s - \frac1\mu \int_0^p (1- S)^{-1}(s) \d s \Big) \\
    = & \sqrt{n} ( \widehat \mu_n^{-1} \cdot (\Phi \circ \Psi \circ (1 -\widehat S))(p) - \mu^{-1} \cdot ( \Phi \circ \Psi \circ (1 - S) )(p)).
\end{align*}
Here the functionals $\Phi$ and $\Psi$ are 
\begin{align*}
 & \Phi: D_-[0,1] \mapsto C[0,1], \quad h \mapsto \Big( p \mapsto \int_0^p h(s) \d s \Big), \quad  \text{and} \\
 & \Psi: D_\uparrow[0,\tau] \mapsto D_-[0,1] , \quad k \mapsto k^{-1} \quad \text{(the left-continuous generalized inverse)}.
\end{align*}
Suppose that $S$ is continuously differentiable on its support with strictly positive derivative $f$, bounded away from zero.
The Hadamard-differentiability of $\Psi$ at $(1-S)$ tangentially to $C[0,\tau]$ then holds according to Lemma~3.9.23 in \cite{vaart96}, p. 386.
Its derivative map is given by $\alpha \mapsto - \frac{\alpha}{f} \circ (1-S)^{-1}$.
The other functional $\Phi$ is obviously Hadamard-differentiable at $S^{-1} \in C[0,1]$ tangentially to $C[0,1]$ 
since $\Phi$ itself is linear and the domain of integration is bounded.
Next, 
$$\sqrt{n}\Big( \frac{1}{\widehat \mu} - \frac{1}{\mu}\Big) = \sqrt{n} ( \Upsilon( \widehat g_n(0)) - \Upsilon( g(0))) $$ 
where $\Upsilon: (0,\infty) \rightarrow (0,\infty), r \mapsto \frac1r$, \
$g(0) = \E[ T - a \ | \ T > 0 ] = \E[T]$ is the mean-residual life-time function at $0$ and $\widehat g_n(0)$ its estimated counterpart. 
Clearly, $\Upsilon$ is (Hadamard-)differentiable 
and the required Hadamard-differentiability of $(1-S) \mapsto g(0)$ follows immediately from Lemma~\ref{lem:hadamard_ex}.
Finally, the multiplication functional is also Hadamard-differentiable.
All in all, we conclude that $W_n = \sqrt{n}( \Xi(\widehat S_n) - \Xi(S))$
for a functional $\Xi : D[0,\tau] \rightarrow C[0,1]$ which is Hadamard-differentiable at $S$
tangentially to $C[0,\tau]$.
Theorem~\ref{thm:ebs_kme_wl} in combination with the functional $\delta$-method (for the bootstrap)
immediately implies that $W_n$ and $W_n^*$ both converge in (conditional) distribution to the same continuous Gaussian process (in outer probability given $\b X$).
Time-simultaneous inference procedures for the Lorenz curve, such as tests for equality and confidence bands are constructed straightforwardly.

Finally, the normalized estimated Gini index allows the representation
\begin{align*}
 \sqrt{n}(\widehat G_n - G) = 2 \sqrt{n}( \{ \Phi \circ \Xi \}(1) \circ \widehat S_n - \{ \Phi \circ \Xi \}(1)  \circ S)
\end{align*}
of which $\{ \Phi \circ \Xi \}(1)$ is again Hadamard-differentiable at $S$
tangentially to $C[0,\tau]$.
Hence, confidence intervals for $G$ with bootstrap-based quantiles are constructed in the same way as before.


\section{Discussion}
\label{sec:EBS_KME_DISC}

In this article we established consistency of the bootstrap for Kaplan-Meier estimators 
on the whole support of the estimated survival function.
By means of the functional delta-method this conditional weak convergence is transferred to Hadamard-differentiable functionals such as the mean-residual lifetime, the Lorenz curve or the Gini index.
Further applications include the expected length of stay in the transient state (e.g. \citealt{grand15})
or the probability of concordance (e.g. \citealt{pocock2012win}, \citealt{dobler16b}).

This bootstrap consistency on the whole support may also be extended to more general inhomogeneous Markovian multistate models.
Based on the martingale representation of Aalen-Johansen estimators for transition probability matrices   (e.g. \citealp{abgk93}, p. 289),
one could try to generalize the results of \cite{gill83} to this setting.
Here the notion of the `largest event times' requires special attention
as these may differ for different types of transitions.
A reasonable first step towards such a generalization would be an analysis in competing risks set-ups
where the support of each cumulative incidence function provides a natural domain to investigate weak convergences on.
Once weak convergence of the estimators on the whole support is verified,
martingale arguments similar to those of \cite{akritas86} and \cite{gill83} 
may be employed in order to obtain such (now conditional) weak convergences for the resampled Aalen-Johansen estimator using a variant of Efron's bootstrap.
In more general Markovian multi-state models we could independently draw with replacement from the sample that contains all individual {\it{trajectories}} rather than single observed transitions in order to not corrupt the dependencies within each individual;
see for example \cite{tattar12} for a similar suggestion.
Applications of this theory could include inference on more refined variants of the probability of concordance or the expected length of stay.
Considering a progressive disease in a two-sample situation, for instance, 
we would like to compare the probability that an individual of group one remains longer in a less severe disease state than an individual of group two.
Accurate inference procedures for the mean residual life-time in a state of disability given any state at present time 
offers another kind of application.

\section*{Acknowledgements}

The author would like to thank Markus Pauly (University of Ulm) for helpful discussions.

\vspace{.5cm}
\noindent {\Large\textbf{Appendix}}

\appendix

Some of the following proofs (Appendix~\ref{sec:EBS_KME_PROOFS}) rely on the ideas of \cite{gill83}.
In order to also apply (variants of) his lemmata in our bootstrap context,
Appendix~\ref{sec:gills_lemmata} below contains all required results.
`Tightness' in the support's right boundary $\tau$ for the bootstrapped Kaplan-Meier estimator
is essentially shown via a bootstrap version of the approximation theorem for truncated estimators as in Theorem~3.2 in \cite{billingsley99}; cf. Appendix~\ref{sec:trunc_tech}.
Define by $Y(u) = n \wh H_{n-}(u)$ the process counting the number of individuals at risk of dying, and by $Y^*(u)$ its bootstrap version.

\section{Proofs}
\label{sec:EBS_KME_PROOFS}
\begin{proof}[Proof of Lemma~\ref{lem:VAR_KME_WL}]
Proof of (a):
Let $t <  \tau$ and suppose~\eqref{cond:weak_cens} holds.
By the continuous mapping theorem and the boundedness away from zero of $\frac{1}{G}$ on $[0,t]$,
it clearly follows that $ -\int_0^{t} \frac{\d \widehat S_n}{\widehat G_{n-}} \oPo - \int_0^{t} \frac{\d S}{G_-}$ as $n \rightarrow \infty$.
Letting $t \uparrow \tau$, the right-hand side converges towards \ $- \int_0^{\tau} \frac{\d S}{G_-} < \infty$.
It remains to apply Theorem~3.2 of \cite{billingsley99} in order to verify the assertion for $t=0$ and hence for all $t \leq \tau$ by the continuous mapping theorem.
Thus, we show that for all $\varepsilon > 0$, 
$$ \lim_{t \uparrow \tau} \limsup_{n \rightarrow \infty} P \Big( - \int_t^{\tau} \frac{\d \widehat S_n}{\widehat G_{n-}} > \varepsilon \Big) = 0 .$$
Let $\widehat T_n$ again be the largest observation among $X_1, \dots, X_n$ and define, for any 
$\beta > 0$, $$B_\beta := \{ \widehat S_n(s) \leq \beta^{-1} S(t) \text{ and } \widehat H_n(s-) \geq \beta H(s-) \text{ for all } s \in [0,\widehat T_n] \}.$$
By Lemmata~\ref{lem:gill83_2.6} and~\ref{lem:gill83_2.7}, the probability
$ p_\beta := 1 - P(B_\beta) \leq \beta + \frac{e}{\beta} \exp(-1 / \beta)$
is arbitrary small for sufficiently small $\beta > 0$.
Hence, by Theorem~1.1 of \cite{stute93} (applied for the concluding convergence),
\begin{align*}
 & P \Big( - \int_t^{\tau} \frac{\d \widehat S_n}{\widehat G_{n-}} > \varepsilon \Big)
 = P \Big( - \int_t^{\tau} \frac{\widehat S_{n-} \d \widehat S_n}{\widehat H_{n-}} > \varepsilon \Big) \\
 & \leq P \Big( - \beta^{-2} \int_t^{\tau} \frac{S_- \d \widehat S_n}{H_-} > \varepsilon \Big) + p_\beta
 \\ & = P \Big( - \beta^{-2} \int_t^{\tau} \frac{\d \widehat S_n}{G_-} > \varepsilon \Big) + p_\beta
 \rightarrow P \Big( - \beta^{-2} \int_t^{\tau} \frac{\d S}{G_-} > \varepsilon \Big) + p_\beta.
\end{align*}
For large $t < \tau$ and by the continuity of $S$, the far right-hand side of the previous display equals $p_\beta$.

Proof of (b): First note that the uniform convergences in probability in Theorems~IV.3.1 and~IV.3.2 of \cite{abgk93}, p. 261ff., yield, for any $\varepsilon > 0$,
\begin{align*}
 \sup_{(u,v) \in [0,\tau-\varepsilon]^2} | \widehat \Gamma_n(u,v) - \Gamma(u,v) | \oPo 0 \quad \text{as } n \rightarrow \infty.
\end{align*}
Further, the dominated convergence theorem and $ S_- \d A = - \d S$ show that
\begin{align*}
 \Gamma(u,v) = - \int_0^\tau \b 1\{ w \leq u \wedge v \} \frac{S(u)S(v)}{S(w-) S(w-)} \frac{\d S(w)}{G(w-)} \rightarrow - \int_0^\tau 0 \frac{\d S}{G_-} = 0
\end{align*}
as $u,v \rightarrow \tau$.
Hence, it remains to verify the remaining condition~(3.8) of Theorem~3.2 in \cite{billingsley99} in order to conclude this proof.
That is, for each positive $\delta$ we show
\begin{align*}
 \lim_{u,v \rightarrow \tau} \limsup_{n \rightarrow \infty} P \Big( \sup_{(u,v) \in [0,\tau]^2} | \widehat \Gamma_n(u,v) - \widehat \Gamma_n(\tau, \tau) | \geq \delta \Big) = 0.
\end{align*}
To this end, rewrite $\widehat \Gamma_n(u,v) - \widehat \Gamma_n(\tau, \tau)$ as
\begin{align*}
 \int_{u \wedge v}^\tau \frac{\widehat S_n(\tau)  \widehat S_n(\tau) }{\widehat S_n(w- ) \widehat S_n(w-)  } \frac{\d \widehat S_n(w)}{\widehat G_{n}(w-)}
 + \int_0^{u \wedge v} \frac{\d \widehat S_n}{\widehat S^2_{n-} \widehat G_{n-}} ( \widehat S_n^2(\tau) - \widehat S_n(u) \widehat S_n(v) ).
\end{align*}
The left-hand integral is bounded in absolute value by $-\int_{u \wedge v}^\tau \frac{\d \widehat S_n}{\widehat G_{n-}}$
which goes to $-\int_{u \wedge v}^\tau \frac{\d S}{G_{-}}$ in probability as $n \rightarrow \infty$ by (a).
For large $u,v$ this is arbitrarily small. 

The remaining integral is bounded in absolute value by 
\begin{align*}
 -\int_0^{u \wedge v} \frac{\widehat S_n(u) \widehat S_n(v)}{\widehat S^2_{n-}} \frac{\d \widehat S_n}{\widehat G_{n-}} 
 = - n^2 \int_0^{u \wedge v} \frac{\widehat S_n(u) \widehat S_n(v)}{Y^2} \widehat G_{n-} \d \widehat S_n.
\end{align*}
By Lemmata~\ref{lem:gill83_2.6} and~\ref{lem:gill83_2.7}
this integral is bounded from above by
\begin{align*}
 - \int_0^{u \wedge v} \frac{S(\widehat T_n \wedge u) S(\widehat T_n \wedge v)}{H^2_-} G_{-} \d \widehat S_n
 =  - \int_0^{u \wedge v} \frac{S(\widehat T_n \wedge u) S(\widehat T_n \wedge v)}{S^2_- G_-} \d \widehat S_n
\end{align*}
on a set with arbitrarily high probability.
For sufficiently large $n$ we also have $\widehat T_n > u \wedge v$ with arbitrarily high probability.
Next, Theorem~1.1 in \cite{stute93} yields
\begin{align*}
 - \int_0^{u \wedge v} \frac{S(u) S(v)}{S^2_- G_-} \d \widehat S_n \oPo - \int_0^{u \wedge v} \frac{S(u) S(v)}{S^2_- G_-} \d S \quad \text{as } n \rightarrow \infty.
\end{align*}
As above the dominated convergence theorem shows the negligibility of this integral as $u,v\rightarrow \tau$.
\end{proof}

\begin{proof}[Proof of Theorem~\ref{thm:ebs_kme_wl}]
  For the proof of weak convergence of the bootstrapped Kaplan-Meier estimator on each Skorohod space $D[0,t]$, $t < \tau$, 
  see e.g. \cite{akritas86}, \cite{lo86} or \cite{horvath87}.
  By defining these processes as constant functions after $t$, 
  the convergences equivalently hold on $D[0,\tau]$.
  This takes care of Condition (a) in Lemma~\ref{lem:abschneiden}, 
  while (c) is obviously fulfilled by the continuity of the limit Gaussian process.

To close the indicated gap for the bootstrapped Kaplan-Meier process on the whole support $[0,\tau]$,
it remains to analyze Condition (b).
This is first verified for the truncated process by following the strategy of \cite{gill83} 
while applying the martingale theory of \cite{akritas86} for the bootstrapped counting processes.
Thus, the truncation technique of Lemma~\ref{lem:abschneiden} shows the convergence in distribution of the truncated process.
Finally, the negligibility of the remainder term is shown similarly as in \cite{ying89}.

We will make use of the fact that our martingales, stopped at arbitrary stopping times, retain the martingale property; cf. \cite{abgk93}, p. 70, for sufficient conditions on this matter.
Similarly to the largest event or censoring time $\widehat T_n$,
introduce the largest bootstrap time $T^*_n = \max_{i=1, \dots, n} X_i^*$,
being an integrable stopping time with respect to the filtration of \cite{akritas86}
who used Theorem~3.1.1 of \cite{gill80}:
Hence, we choose the filtration given by
\begin{align*}
 \mathcal{F}_t := \{ X_i, \delta_i, \delta_i^* \bs 1\{X_i^* \leq t\}, X_i^*\bs 1\{X_i^* \leq t\}: i = 1, \dots, n  \}, \quad 0 \leq t \leq \tau;
\end{align*}
see also \cite{gill80}, p. 26, for a similar minimal filtration.
Note that we did not include the indicators $\bs 1\{X_i^* \leq t\}$ into the filtration since
their values are already determined by all the $X_i^*\bs 1\{X_i^* \leq t\}$:
According to our assumptions, $X_i^* > 0$ a.s. for all $i=1,\dots,n$.

We would first like to verify condition (b) in Lemma~\ref{lem:abschneiden}
for the stopped bootstrap Kaplan-Meier process.
That is, for each $\varepsilon > 0$ and an arbitrary subsequence $(n') \subset (n)$ there is another subsequence $(n'') \subset (n')$ such that 
\begin{align}
\nonumber
 & \lim_{t \uparrow \tau} \limsup_{n'' \rightarrow \infty} 
 P( \sup_{t \leq s < T_n^*} \sqrt{n''} | (S_n^* - \widehat S_n)(s) - (S_{n}^* - \widehat S_n)(t) | > \varepsilon \ | \ \b X ) \\
 \nonumber
 & \leq \lim_{t \uparrow \tau} \limsup_{n'' \rightarrow \infty} 
 P( \sup_{t \leq s < \widehat T_n} \sqrt{n''} | (S_{n}^* - \widehat S_n)(s \wedge T_n^*) - (S_{n}^* - \widehat S_n)(t \wedge T^*_n) | > \varepsilon \ | \ \b X ) \\
 & \quad = 0 \quad  \text{a.s.} \label{eq:clipping}
 \end{align}
for all $\varepsilon > 0$.
Here $\sigma(\b X) = \mc F_0$ summarizes the collected data.
Due to the boundedness away from zero, i.e. $\inf_{t \leq s < \widehat T_n} \widehat S_n(s) > 0$,
we may rewrite the bootstrap process $\sqrt{n} ( S_{n}^* - \widehat S_n ) (s) = \sqrt{n} \Big( \frac{S_{n}^*(s)}{\widehat S_n(s)} - 1 \Big) \widehat S_n(s)$
for each $s \in [t,\widehat T_n)$
of which the bracket term is a square integrable martingale; see \cite{akritas86} again.
Hence, the term $\sqrt{n}  (S_n^* - \widehat S_n)(s)$ in~\eqref{eq:clipping} equals 
\begin{align}
  \label{eq:clipping sup}
 M_n^*(s)  \widehat S_n(s \wedge T^*_n) :=
\sqrt{n} \Big( \frac{S_{n}^*(s \wedge T^*_n)}{\widehat S_n(s \wedge T^*_n)} - 1 \Big) \widehat S_n(s \wedge T^*_n),
\end{align}
whereof $(M^*_n(s))_{s \in [0,\widehat T_n)}$ is again a square integrable martingale.
Indeed, its predictable variation process evaluated at the stopping time $s = T_n^*$ is finite
(having the sufficient condition of \cite{abgk93}, p. 70, for a stopped martingale to be a square integrable martingale in mind):
The predictable variation is given by
\begin{align*}
 s \mapsto \langle M_n^* \rangle (s) = \int_0^{s \wedge T_n^*} \Big( \frac{S_{n-}^*}{\widehat S_n} \Big)^2 
  \frac{(1- \Delta \widehat A_n) \d \widehat A_n}{H^*_{n-}},
\end{align*}
where $H^*_n$ is the empirical survival function of $X_1^*, \dots, X_n^*$ and $\Delta f$ denotes the increment process $s \mapsto f(s+) - f(s-)$ of a monotone function $f$.
The supremum in~\eqref{eq:clipping} is bounded by
\begin{align*}
 \sup_{t \leq s < \widehat T_n} | M^*_n(s) - M^*_n(t) | \widehat S_n(s \wedge T^*_n) 
  + \sup_{t \leq s < \widehat T_n} | M^*_n(t) | | \widehat S_n(s \wedge T^*_n) - \widehat S_n(t \wedge T^*_n) |
\end{align*}
of which the right-hand term is not greater than $ | M_n^*(t) | \widehat S_n(t \wedge T^*_n) $.
By the convergence in distribution of the bootstrapped Kaplan-Meier estimator on each $D[0,\widetilde \tau]$, $\widetilde \tau < \tau$,
we have convergence in conditional distribution of $ M^*_n(t) \widehat S_n(t \wedge T^*_n) $ given $\b X$ 
towards $N(0, S^2(t) \Gamma(t,t))$ in probability. 
Hence, $$ \lim_{n'' \rightarrow \infty} P ( | M_n^*(t) \widehat S_n(t \wedge T^*_n) | > \varepsilon/2 \ | \ \b X) \rightarrow 1 - N(0, \Gamma(t,t))(-\varepsilon/2, \varepsilon/2) $$ 
almost surely along subsequences $(n'')$ of arbitrary subsequences $(n') \subset (n)$.
Since the variance of the normal distribution in the previous display goes to zero as $t \uparrow \tau$, cf. (2.4) in \cite{gill83},
the above probability vanishes as $t \uparrow \tau$.

By Lemma~\ref{lem:gill83_2.9}, the remainder $\sup_{t \leq s < \widehat T_n} | M^*_n(s) - M^*_n(t) | \widehat S_n(s \wedge T^*_n) $ is not greater than
\begin{align}
 2 \sup_{t \leq s < \widehat T_n} \Big| \int_t^s \widehat S_n(u) \d M_n^*(u) \Big| .
\end{align}
Since, given $\b X$, $\widehat S_n$ is a bounded and predictable process,
this integral is a square integrable martingale on $[t,\widehat T_n)$.
We proceed as in \cite{gill83} by applying Leng\-lart's inequality, cf. Section~II.5.2 in \cite{abgk93}:
For each $\eta > 0$ we have
\begin{align}
\begin{split}
\label{eq:lenglart}
  & P \Big( \sup_{t \leq s < T^*_n \wedge \tau} \Big| \int_t^s \widehat S_n \d M^*_n \Big| > \varepsilon \ \Big| \ \b X \Big) \\
 & \quad \leq \frac{\eta}{\varepsilon^2} + P \Big( \Big| \int_t^{\tau \wedge T^*_n} S_{n-}^{*2} 
      \frac{(1- \Delta \widehat A_n) \d \widehat A_n}{H^*_{n-}} \Big| 
	> \eta \ \Big| \ \b X \Big).
\end{split}
\end{align}
We intersect the event on the right-hand side of~\eqref{eq:lenglart} with 
$B_{H,n,\beta}^* := \{H^*_n(s-) \geq \beta \widehat H_n(s-)$ for all $s \in [t,T^*_n]\}$ 
and also with $B_{S,n,\beta}^* :=  \{S^*_n(s) \leq \beta^{-1} \widehat S_n(s)$ for all $s \in [t,T^*_n]\}$.
According to Lemmata~\ref{lem:gill83_2.6} and~\ref{lem:gill83_2.7}, 
the conditional probabilities of these events are at least 
$ 1- \exp(1-1/\beta)/\beta$ and $1-\beta$, respectively, for any $\beta \in (0,1)$.
Thus, \eqref{eq:lenglart} is less than or equal to
\begin{align}
\label{eq:beta} 
\frac{\eta}{\varepsilon^2} + \beta + \frac{\exp(1-1/\beta)}{\beta} + 
    \b 1 \Big\{ \beta^{-3} \Big| \int_t^{\tau \wedge \widehat T_n} \widehat S_{n-}^{2} 
      \frac{(1- \Delta \widehat A_n) \d \widehat A_n}{\widehat H_{n-}} \Big| 
	> \eta \Big\}.
\end{align}
In order to show the almost sure negligibility of the indicator function as $n \rightarrow \infty$ and then $t \uparrow \tau$,
we analyze the corresponding convergence of the integral.
Since $- \d \widehat S_n = \widehat S_{n-} \d \widehat A_n$, the integral is less than or equal to 
$$-\int_t^{\tau} \frac{\widehat S_{n-} \d \widehat S_{n}}{\widehat H_{n-}} = -\int_t^{\tau} \frac{\d \widehat S_{n-}}{\widehat G_{n-}}.$$
Lemma~\ref{lem:VAR_KME_WL} implies that for each subsequence $(n') \subset (n)$ there is another subsequence $(n'') \subset (n')$ 
such that $-\int_t^{\tau} \frac{\d \widehat S_{n-}}{\widehat G_{n-}} \rightarrow - \int_t^{\tau} \frac{\d S}{G_-}$ a.s. for all $t \in [0, \tau] \cap \Q$ along $(n'')$.
Due to $P(Z_1 \in \Q) = 0$, the same convergence holds for all $ t \leq  \tau$.
Letting now $t \uparrow \tau$ shows that the indicator function in \eqref{eq:beta} vanishes almost surely in limit superior along $(n'')$ if finally $t \uparrow\tau$.
The remaining terms are arbitrarily small for sufficiently small $\eta, \beta > 0$.
Hence, all conditions of Lemma~\ref{lem:abschneiden} are met and the assertion follows for the stopped process
$$ ( \b 1 \{ s < T^*_n\} \sqrt{n}(S_n^*(s) - \widehat S_n(s)) 
  + \b 1 \{ s \geq  T^*_n\} \sqrt{n}(S_n^*(T_n^*-) - \widehat S_n(T_n^*-)))_{s \in [0,\tau]}.$$

Finally, we show the asymptotic negligibility of
\begin{align*}
 \sup_{T^*_n \leq s \leq \tau} \sqrt{n} | S^*_{n}(s) - \widehat S_n(s) | 
 & \leq \sup_{T^*_n \leq s \leq \tau} \sqrt{n} ( S^*_{n}(s) +  \widehat S_n(s)) \\
 & = \sqrt{n} S^*_{n}(T^*_n) + \sqrt{n} \widehat S_n(T^*_n);
\end{align*}
cf. \cite{ying89} for similar considerations.
Again by Lemma~\ref{lem:gill83_2.6},
we have for any $\varepsilon > 0, \beta \in (0,1)$ that
\begin{align*}
 & P ( \sqrt{n} S^*_{n}(T^*_n) + \sqrt{n} \widehat S_n(T^*_n) > \varepsilon \ | \ \b X ) \\
 & \leq P ( \sqrt{n} S^*_{n}(T^*_n) > \varepsilon / 2 \ | \ \b X ) + P(\sqrt{n} \widehat S_n(T^*_n) > \varepsilon / 2 \ | \ \b X ) \\
 & \leq P ( \sqrt{n} \widehat S_n(T^*_n) > \beta \varepsilon / 2 \ | \ \b X ) + P(\sqrt{n} \widehat S_n(T^*_n) > \varepsilon / 2 \ | \ \b X ) + \beta.
\end{align*}
Define the generalized inverse $\widehat S_n^{-1}(u) := \inf\{ s \leq \tau: \widehat S_n(s) \geq u\}$.
The independence of the bootstrap drawings as well as arguments of quantile transformations yield
\begin{align*}
 P(\sqrt{n} \widehat S_n(T^*_n) > \varepsilon \ | \ \b X )
 & = P(X_1^* < \widehat S_n^{-1}( \varepsilon / \sqrt{n}) \ | \ \b X )^n \\
 & = \Big[ 1 - \frac1n | \{ i : X_i \geq \widehat S_n^{-1}(\varepsilon / \sqrt{n} ) \} | \Big]^{n}.
\end{align*}
The cardinality in the display goes to infinity in probability, and hence almost surely along subsequences.
Indeed, for any constant $C >0$,
\begin{align*}
 & P ( | \{ i : X_i \geq \widehat S_n^{-1}(\varepsilon / \sqrt{n} ) \} | \geq C ) \\
 & = P ( | \{ i : \widehat S_n(X_i) \geq \varepsilon / \sqrt{n} \} | \geq C ) \\
 & \geq P ( | \{ i : \widehat H_n(X_i) \geq \varepsilon / \sqrt{n} \} | \geq C ) \\
 & = P \Big( \Big| \Big\{i : \frac{i-1}{n} \geq \frac{\varepsilon}{\sqrt{n}} \Big\} \Big| \geq C \Big) \\
 & = \b 1 \Big\{ \Big| \Big\{i : \frac{i-1}{n} \geq \frac{\varepsilon}{\sqrt{n}} \Big\} \Big| \geq C \Big\}
 = \b 1 \{ | \{ \lceil \varepsilon \sqrt{n} \rceil + 1, \dots, n \} |  \geq C\}.
\end{align*}
Clearly, this indicator function goes to 1 as $n \rightarrow \infty$. 
\end{proof}

\begin{proof}[Proof of Lemma~\ref{lem:ebs_kme_var_wl}]
 For the most part, we follow the lines of the above proof of Lemma~\ref{lem:VAR_KME_WL}
 by verifying condition~(3.8) of Theorem~3.2 in \cite{billingsley99}.
 To point out the major difference to the previous proof,
 we consider
 \begin{align*}
  - \int_{u \wedge v}^\tau \frac{\d S^*_{n}}{G^*_{n-}} 
  = \int_{u \wedge v}^\tau \frac{S^*_{n-} \d A_n^*}{G^*_{n-}}
  = \int_{u \wedge v}^\tau \frac{S^*_{n-}}{G^*_{n-}} J^* \d ( A_n^* - \widehat A_n) 
    + \int_{u \wedge v}^\tau \frac{S^*_{n-}}{G^*_{n-}} J^* \d \widehat A_n,
 \end{align*}
 where $J^*(u) = \b 1\{ Y^*(u) > 0 \}$.
  The arguments of \cite{akritas86} show that 
  $\int_{u \wedge v}^\cdot \frac{S^*_{n-}}{G^*_{n-}} J^* \d ( A_n^* - \widehat A_n) $
  is a square-integrable martingale with predictable variation process given by
  \begin{align*}
   t \longmapsto \int_{u \wedge v}^t \frac{S^{*2}_{n-}}{G^{*2}_{n-}} \frac{J^*}{Y^*} (1 - \Delta \widehat A_n) \d \widehat A_n.
  \end{align*}
  After writing $S_n^* G^*_n = H_n^*$, a two-fold application of Lemmata~\ref{lem:gill83_2.6} and~\ref{lem:gill83_2.7} (at first to the bootstrap quantities $S_n^*$ and $H_n^*$, then to the Kaplan-Meier estimators $\widehat S_n$ and $\widehat H_n$)
  show that the predictable variation in the previous display is bounded from above by
  \begin{align*}
    - \beta^{-13} \frac1n \int_{u \wedge v}^t \frac{S^{3}_-}{H^{3}_{-}} \d \widehat S_n
    = - \beta^{-13} \frac1n \int_{u \wedge v}^t \frac{\d \widehat S_n}{G^{3}_{-}}
  \end{align*}
  on a set with arbitrarily large probability depending on $\beta \in (0,1)$.
  Here we also used that $\widehat S_{n-} \d \widehat A_n = \d \widehat S_n$.
  Due to \eqref{cond:weaker_cens}, Theorem~1.1 of \cite{stute93} yields
  $$-\int_{u \wedge v}^t \frac{\d \widehat S_n}{G^{3}_{-}} \oPo - \int_{u \wedge v}^t \frac{\d S}{G^{3}_{-}} < \infty $$
  and hence the asymptotic negligibility of the predictable variation process in probability. 
  By Rebolledo's theorem (Theorem II.5.1 in \citealp{abgk93}, p. 83), $\int_{u \wedge v}^\tau \frac{S^*_{n-}}{G^*_{n-}} J^* \d ( A_n^* - \widehat A_n) $ hence goes to zero in conditional probability.
  The remaining integral $ \int_{u \wedge v}^\tau \frac{S^*_{n-}}{G^*_{n-}} J^* \d \widehat A_n$ 
  is treated similarly with Lemmata~\ref{lem:gill83_2.6} and~\ref{lem:gill83_2.7}
  and Theorem~1.1 of \cite{stute93}
  yielding a bound in terms of $\int_{u \wedge v}^\tau \frac{\d S}{G_-} $.
  This is arbitrarily small for sufficiently large $u,v < \tau$.
\end{proof}

\begin{proof}[Proof of Lemma~\ref{lem:hadamard_ex}]
 Proof of (b): Throughout, the functional spaces $D[t_1,\tau]$ and $\widetilde D[t_1,t_2]$ are equipped with the supremum norm.
  For some sequences $t_n \downarrow 0$ and $h_n \rightarrow h$ in $D[t_1,\tau]$ such that $\theta + t_n h_n \in \widetilde D[t_1,t_2]$, 
  consider the supremum distance
  \begin{align}
  \label{eq:sup_dist_MRLT}
   \sup_{s \in [t_1,t_2]} \Big| \frac{1}{t_n} [ \psi(\theta + t_n h_n)(s) - \psi(\theta)(s) ] - (\d \psi(\theta) \cdot h)(s)  \Big|.
  \end{align}
  The proof is concluded if~\eqref{eq:sup_dist_MRLT} goes to zero.
  For an easier access the expression in the previous display is first analyzed for each fixed $s \in [t_1,t_2]$:
  \begin{align}
  \label{eq:hadamard_MRLT}
   \nonumber & \frac{1}{t_n} [ \psi(\theta + t_n h_n)(s) - \psi(\theta)(s) ] - (\d \psi(\theta) \cdot h)(s)  \\ 
   \nonumber & = \frac{1}{t_n} \frac{1}{\theta(s) + t_n h_n(s)} \frac{1}{\theta(s)} \\
   \nonumber & \quad \times \Big[\theta(s) \int_s^\tau (\theta(u) + t_n h_n(u) ) \d u
    - (\theta(s) + t_n h_n(s) ) \int_s^\tau \theta(u) \d u\Big] \\
   \nonumber & \quad  - \frac{1}{\theta(s)} \int_s^\tau h(u) \d u + h(s) \int_s^\tau \frac{\theta(u)}{\theta^2(s)} \d u \\
   \nonumber & = \frac{1}{\theta(s) + t_n h_n(s)} \int_s^\tau h_n(u)  \d u 
    - \frac{h_n(s)}{\theta(s) + t_n h_n(s)} \frac{1}{\theta(s)} \int_s^\tau \theta(u) \d u \\
   \nonumber & \quad  - \frac{1}{\theta(s)} \int_s^\tau h_n(u) \d u + h_n(s) \int_s^\tau \frac{\theta(u)}{\theta^2(s)} \d u \\
   \nonumber & \quad  - \frac{1}{\theta(s)} \int_s^\tau (h(u) - h_n(u)) \d u - (h(s) - h_n(s)) \int_s^\tau \frac{\theta(u)}{\theta^2(s)} \d u   \\
   \nonumber & = - \int_s^\tau h_n(u)  \d u \frac{t_n h_n(s)}{[\theta(s) + t_n h_n(s)] \theta(s)} 
    + h_n(s) \int_s^\tau \frac{\theta(u)}{\theta^2(s)} \d u \frac{t_n h_n(s)}{\theta(s) + t_n h_n(s)} \\
    & \quad  - \frac{1}{\theta(s)} \int_s^\tau (h(u) - h_n(u)) \d u - (h(s) - h_n(s)) \int_s^\tau \frac{\theta(u)}{\theta^2(s)} \d u.
  \end{align}
  For large $n$, each denominator is bounded away from zero:
  To  see this, denote $\varepsilon := \inf_{s \in [t_1,t_2]}|\theta(s)|$ and 
  $C  := \sup_{s \in [t_1,t_2]} |h(u)|$.
  Thus,   $$ \sup_{s  \in [t_1,t_2]} | h_n(s) | \leq \sup_{s  \in [t_1,t_2]} | h_n(s) - h(u) | + \sup_{s  \in [t_1,t_2]} | h(s) | \leq \varepsilon + C$$ for each $n$ large enough.
  It follows that, for each such $n$ additionally satisfying $t_n \leq \varepsilon(2\varepsilon + 2C)^{-1}$,
  the denominators are bounded away from zero, in particular, $\inf_{s \in [t_1,t_2]}|\theta(s) + t_n h_n(s)| \geq \varepsilon / 2$.
  Thus, taking the suprema over $s \in [t_1,t_2]$, 
  the first two terms in~\eqref{eq:hadamard_MRLT} become arbitrarily small by letting $t_n$ be sufficiently small.
  The remaining two terms converge to zero since $\sup_{s  \in [t_1,t_2]} | h_n(s) - h(s) | \rightarrow 0$
  and $\sup_{u \in [t_1,\tau]} | \theta(u) | < \infty$.
  Note here that $$\int_{t_1}^\tau | h(u) - h_n(u) | \d u \leq \sup_{u \in [t_1,\tau]} | h(u) - h_n(u) | (\tau - t_1) \rightarrow 0 $$
  due to $\tau < \infty$.
\end{proof}

\begin{proof}[Proof of Lemma~\ref{lem:EBS_KME_APP}]
 The convergences are immediate consequences of the functional delta-method, Theorem~\ref{thm:ebs_kme_wl}
 and the bootstrap version of the delta-method; cf. Section~3.9 in \cite{vaart96}.
 Simply note that all considered survival functions are elements of $D_{< \infty} \cap \widetilde D[t_1,t_2]$ 
 (on increasing sets with probability tending to one) and
 that the survival function of the life-times is assumed continuous and bounded away from zero on compact subsets of $[0,\tau)$.
 Further, there is a version of the limit Gaussian processes with almost surely continuous sample paths.
 
 For the representation of the variance of the limit distribution in part (a) we refer to \cite{vaart96}, p. 383 and 397.
 The asymptotic covariance structure in part (b) is easily calculated using Fubini's theorem 
 -- for its applicability note that the variances $\Gamma(r,r)$ of the limit process $W$ of the Kaplan-Meier estimator exist at all points of time $r \in [0,\tau]$.
 Thus, since $W$ is a zero-mean process, we have for any $0 \leq r \leq s < \tau$,
 \begin{align*}
  & cov \Big( \int_r^\tau \frac{W(u)}{S(r)} \d u - \int_r^\tau \frac{W(r) S(u)}{S^2(r)} \d u, 
    \int_s^\tau \frac{W(v)}{S(s)} \d v - \int_s^\tau \frac{W(s) S(v)}{S^2(s)} \d v \Big)  \\
  &  = \int_r^\tau \int_s^\tau \Big[ \Gamma(u,v) - \frac{S(u)}{S(r)} \Gamma(r,v) 
    - \frac{S(v)}{S(s)} \Gamma(s,u) + \frac{S(u) S(v)}{S(r)S(s)} \Gamma(r,s) \Big] \frac{ \d u \d v}{S(r)S(s)} .
 \end{align*}
 Inserting the definition $\Gamma(r,s) = S(r) S(s) \sigma^2(r \wedge s)$
 and splitting the first integral into $\int_r^\tau = \int_r^s + \int_s^\tau$ yields that the last display equals
 \begin{align*}
  & \int_r^\tau \int_s^\tau \frac{S(u)S(v)}{S(r)S(s)} [ \sigma^2(u \wedge v) - \sigma^2(r \wedge v) - \sigma^2(s \wedge u) + \sigma^2(r \wedge s) ] \d u \d v \\
  & =  \int_s^\tau \int_s^\tau \frac{S(u)S(v)}{S(r)S(s)} [ \sigma^2(u \wedge v) - \sigma^2(r) - \sigma^2(s) + \sigma^2(r) ] \d u \d v \\
  & \quad + \int_r^s \int_s^\tau \frac{S(u)S(v)}{S(r)S(s)} [ \sigma^2(u) - \sigma^2(r) - \sigma^2(u) + \sigma^2(r) ] \d u \d v \\
  & = \int_s^\tau \int_s^\tau \frac{\Gamma(u,v)}{S(r)S(s)} \d u \d v - \sigma^2(r \vee s) g(r) g(s).
 \end{align*}
\end{proof}

\begin{proof}[Proof of Theorem~\ref{thm:EBS_KME_APP}]
 The theorem follows from Lemma~\ref{lem:EBS_KME_APP} combined with the continuous mapping theorem applied to the supremum functional $D[t_1,t_2] \rightarrow \R$, $f \mapsto \sup_{t \in [t_1,t_2]} | f(t) |$
 which is continuous on $C[t_1,t_2]$.
 For the connection between the consistency of a bootstrap distribution of a real statistic and the consistency of the corresponding tests (and the equivalent formulation in terms of confidence regions),
 see Lemma~1 in \cite{janssen03}.
\end{proof}

\section{Adaptations of Gill's (1983) Lemmata}
\label{sec:gills_lemmata}

Abbreviate again the sigma algebra containing all the information of the original sample as $\b X := \sigma(X_i,\delta_i: i = 1,\dots,n)$.
The proofs in Appendix~\ref{sec:EBS_KME_PROOFS} rely on bootstrap versions of Lemmata~2.6, 2.7 and  2.9 in \cite{gill83}.
Since those are stated under the assumption of a continuous distribution function $S$,
but ties in the bootstrap sample are inevitable,
these lemmata need a slight extension.
For completeness, parts (a) of the following two Lemmata correspond to the original Lemmata~2.6 and 2.7 in \cite{gill83}. 
\begin{lemma}[Extension of Lemma~2.6 in \citealp{gill83}]
 \label{lem:gill83_2.6}
For any $\beta \in (0,1)$, 
 \begin{itemize}
  \item[\tn{(a)}] $P( \widehat S_n(t) \leq \beta^{-1} S(t) \text{ for all } t \leq \widehat T_n) \geq 1 - \beta$,
  \item[\tn{(b)}] $P( S^*_n(t) \leq \beta^{-1} \widehat S_n(t) \text{ for all } t \leq T^*_n \ | \ \b X) \geq 1 - \beta$ almost surely.
 \end{itemize}
\end{lemma}
\begin{proof}[Proof of \tn{(b)}]
 All equalities and inequalities concerning conditional expectations are understood as to hold almost surely.
 As in the proof of Theorem~\ref{thm:ebs_kme_wl}, $(S^*_n(t \wedge T^*_n) / \widehat S_n(t \wedge T^*_n) )_{t \in [0, \widehat T_n)}$ defines a right-continuous martingale for each fixed $n$
 and for almost every given sample $\b X$.
 Hence, Doob's $L_1$-inequality (e.g. \citealp{revuz99}, Theorem~1.7 in Chapter~II) yields for each $\beta \in (0,1)$
 \begin{align*}
  & P( \sup_{t \in [0,\widehat T_n)} S_n^*(t \wedge T^*_n) / \widehat S_n(t \wedge T^*_n) \geq \beta^{-1} \ | \ \b X) \\
  & \quad \leq \beta \sup_{t \in [0,\widehat T_n)} E( S_n^*(t \wedge T^*_n) / \widehat S_n(t \wedge T^*_n)  \ | \ \b X) \\
  & \quad = \beta E( S_n^*(0) / \widehat S_n(0)  \ | \ \b X) = \beta.
 \end{align*}
  This implies $P( S_n^* \leq \beta^{-1} \widehat S_n \text{ on } [0,T_n^*) \ | \ \b X) \geq 1 - \beta.$
  It remains to extend this result to the interval's endpoint.
  If the observation corresponding to $T^*_n$ is uncensored, we have $0 = S^*_n(T^*_n) \leq \beta^{-1} \widehat S_n(T^*_n)$.
  Else, the event of interest $\{ S_n^* \leq \beta^{-1} \widehat S_n \text{ on } [0,T_n^*) \}$ (given $\b X$) implies
  that $$S_n^*(T^*_n) = S^*_n(T_n^* - ) \leq \beta^{-1} \widehat S_n(T^*_n -) = \beta^{-1} \widehat S_n(T^*_n).$$
  Thus, for given $\b X$,
  $ \{ S_n^*(T^*_n) \leq \beta \widehat S_n(T^*_n) \} \subset \{ S^*_n \leq \beta \widehat S_n \text{ on } [0,T^*_n) \}$.
\end{proof}
\begin{lemma}[Extension of Lemma~2.7 in \citealp{gill83}]
 \label{lem:gill83_2.7}
For any $\beta \in (0,1)$, 
 \begin{itemize}
  \item[\tn{(a)}] $P( \widehat H_n(t-) \geq \beta H_n(t-) \ \text{ for all } \ t \leq \widehat T_n) \geq 1 - \frac{e}{\beta} \exp(- 1/ \beta)$,
  \item[\tn{(b)}] $P( H_n^*(t-) \geq \beta \widehat H_n(t-) \ \text{ for all } \ t \leq T^*_n \ | \ \b X) \geq 1 - \frac{e}{\beta} \exp(- 1/ \beta)$ almost surely.
 \end{itemize}
\end{lemma}
\begin{proof}[Proof of \tn{(a)}]
As pointed out by \cite{gill83}, the assertion follows from the  inequality for the uniform distribution in Remark~1(ii) of \cite{wellner78}.
By using quantile transformations, his inequality can be shown to hold for random variables having an arbitrary, even discontinuous distribution function.
\\{\it Proof of \tn{(b)}.}
 Fix $X_i(\omega),\delta_i(\omega), i=1,\dots,n$.
 Since $H$ in part (a) is allowed to have discontinuities, (b) follows from (a) for each $\omega$.
\end{proof}
Let $a, b \in D[0,\tau]$ be two (stochastic) jump processes, i.e. processes being constant between two discontinuities.
If $b$ has bounded variation,
we define the integral of $a$ with respect to $b$ via
$$\int_0^s a \d b = \sum a(t) \Delta b(t), \quad s \in (0,\tau],$$
where the sum is over all discontinuities of $b$ inside the interval $(0,s]$.
If $a$ has bounded variation,
we define the above integral via integration by parts:
$\int_0^s a \d b = a(s) b(s) - a(0) b(0) - \int_0^s b_- \d a$.
\begin{lemma}[Adaptation of Lemma~2.9 in \citealp{gill83}]
 \label{lem:gill83_2.9}
Let $h \in D[0,\tau]$ be a non-negative and non-increasing jump process such that $h(0)=1$ and let $Z \in D[0,\tau]$ be a jump process which is zero at time zero.
Then for all $t \leq \tau$,
$$ \sup_{s \in [0,t]} h(s) | Z(s) | \leq 2 \sup_{s \in [0,t]} \Big| \int_0^s h(u) \d Z(u) \Big|. $$
\end{lemma}
\begin{proof}
 The original proof of Lemma~2.9 in \cite{gill83} still applies for the most part with the assumptions of this lemma.
 For the sake of completeness, we present the whole proof.
 
 Let $U(t) = \int_0^t h(s) \d Z(s) $ with a $t \leq \tau$ such that $h(t) > 0$.
 Then
 \begin{align*}
  & Z(t) = \int_0^t \frac{\d U(s)}{h(s)} = \frac{U(t)}{h(t)} - \int_0^t U(s-) \d \Big( \frac{1}{h(s)} \Big) \\
  & \quad  = \int_0^t ( U(t) - U(s-)) \d \Big( \frac{1}{h(s)} \Big)  
    + \frac{U(t)}{h(0)}.
 \end{align*}
  Thus, following the lines of the original proof,
  \begin{align*}
   & | h(t) Z(t) | \leq \Big| \int_0^t (U(t) - U(s-)) \d \Big( \frac{h(t)}{h(s)} \Big) \Big| + | U(t) | h(t) \\
   & \leq 2 \sup_{0 < s \leq t } |U(s)| \Big( 1 - \frac{h(t)}{h(0)} \Big) + \sup_{0 < s \leq t} | U(s) | h(t)
   \leq 2 \sup_{0 < s \leq t } |U(s)|.
  \end{align*}
\end{proof}

\section{Bootstrap Version of the Truncation Technique for Weak Convergence}
\label{sec:trunc_tech}

The following lemma is a conditional variant of Theorem~3.2 in \cite{billingsley99}.
Let $\rho$ be the modified Skorohod metric $J_1$ on $D[0,\tau]$ as in \cite{billingsley99},
i.e. $\rho(f,g) = \inf_{\lambda \in \Lambda} ( \| \lambda \|^o \vee \sup_{t \in [0, \tau]} | f(t) - g(\lambda(t)) | ) $,
where $\Lambda$ is the collection of non-decreasing functions onto $[0, \tau]$ and $\| \lambda \|^o = \sup_{s \neq t} \Big| \log \frac{\lambda(s) - \lambda(t)}{s - t} \Big|$.
For an application in the proof of Theorem~\ref{thm:ebs_kme_wl}, note that $\rho(f,g) \leq \sup_{t \in [0,\tau]} | f(t) - g(t) |$.
\begin{lemma}
 \label{lem:abschneiden}
 Let $X: (\Omega, \mathcal{A}, P)  \rightarrow (D[0,\tau], \rho)$ be a stochastic process 
 and let the sequences of stochastic processes $X_{un}$ and $X_n$
 satisfy the following convergences given a $\sigma$-algebra $\mathcal{C}$:
 \begin{itemize}
  \item[(a)] $X_{un} \oDo Z_u$ given $\mathcal{C}$ in probability as $n \rightarrow \infty$ for every fixed $u$,
  \item[(b)] $Z_{u} \oDo X$ given $\mathcal{C}$ in probability as $u \rightarrow \infty$,
  \item[(c)] for all $\varepsilon > 0$ and for each subsequence $(n') \subset (n)$ there exists another subsequence $(n'') \subset (n')$ such that
    $$ \lim_{u \rightarrow \infty} \limsup_{n'' \rightarrow \infty} P( \rho(X_{u n''}, X_{n''}) > \varepsilon \ | \ \mathcal{C} ) = 0 
    \quad \text{almost surely}. $$
 \end{itemize}
 Then, $X_{n} \oDo X$ given $\mathcal{C}$ in probability as $n \rightarrow \infty$.
\end{lemma}
\begin{proof}
 Choose a sequence $\varepsilon_m \downarrow 0$.
 Let $(n') \subset (n)$ be an arbitrary subsequence and choose subsequences $(n''(\varepsilon_m)) \subset (n')$ and $(u') \subset (u)$
 such that (a) and (b) hold almost surely  and also such that (c) holds along these subsequences.
 Replace $(n''(\varepsilon_m))$ by their diagonal sequence $(n'')$ ensuring (c) simultaneously for all $\varepsilon_m$.
 Let $F \subset D[0, \tau]$ be a closed subset and let $F_{\varepsilon_m} = \{f \in D[0, \tau]: \rho(f, F) \leq \varepsilon_m \}$ 
 be its closed $\varepsilon_m$-enlargement.
 We proceed as in the proof of Theorem~3.2 in \cite{billingsley99} whereas all inequalities now hold almost surely.
 \begin{align*}
  P(X_{n''} \in F \ | \ \mathcal{C}) \leq P(X_{u' n''} \in F_{\varepsilon_m} \ | \ \mathcal{C}) 
      + P(\rho(X_{u' n''}, X_{n''}) > \varepsilon_m \ | \ \mathcal{C}).
 \end{align*}
 The Portmanteau theorem in combination with (a) yields
 \begin{align*}
   \limsup_{n'' \rightarrow \infty} P(X_{n''} \in F \ | \ \mathcal{C}) & \leq P(Z_{u'} \in F_{\varepsilon_m} \ | \ \mathcal{C}) 
       + \limsup_{n'' \rightarrow \infty} P(\rho(X_{u' n''}, X_{n''}) > \varepsilon_m \ | \ \mathcal{C}) .
 \end{align*}
 Condition (b) and another application of the Portmanteau theorem imply that
 \begin{align*}
  \limsup_{n'' \rightarrow \infty} P(X_{n''} \in F \ | \ \mathcal{C}) \leq P(X \in F_{\varepsilon_m} \ | \ \mathcal{C}) .
 \end{align*}
 Let $m \rightarrow \infty$ to deduce
 $
  \limsup_{n'' \rightarrow \infty} P(X_{n''} \in F \ | \ \mathcal{C}) \leq P(X \in F \ | \ \mathcal{C})
 $ 
 almost surely.
 Thus, a final application of Portmanteau theorem as well as the subsequence principle
 lead to the conclusion that $X_{n} \oDo X$ given $\mathcal{C}$ in probability.
\end{proof}


\bibliography{literatur}
\bibliographystyle{plainnat}

\end{document}